\documentclass[12pt,twoside,reqno]{article}
 \usepackage[T1]{fontenc}
\usepackage{latexsym}
\usepackage{amssymb,amsmath,amsthm,bm,setspace}
\usepackage{pdfsync}
\usepackage{hyperref}
\usepackage{mathrsfs,bm,amssymb,color}
\usepackage{color}
\usepackage{graphicx}

\newcommand{\DEL}{\Delta}

\newcommand{\RR}{\mathbb{R}}
\newcommand{\BR}{\mathbb{R}^{n}}
\newcommand{\CC}{\mathbb{C}}

\newcommand{\lka}{\langle}
\newcommand{\rka}{\rangle}

  \newtheorem{theorem}{Theorem}[section]
\newtheorem{definition}[theorem]{Definition}

  \newtheorem{lemma}[theorem]{Lemma}

\renewcommand{\Re}{\text{\rm Re}\,}

\emergencystretch=\maxdimen
\theoremstyle{definition}

\numberwithin{equation}{section}

\title{Gegenbauer kernel filtration on  the unit hypersphere}
\begin{document}
\author{Louis Omenyi$^{1}$\footnote{Corresponding author.}\\
{$^{1 \ast}$Email: omenyi.louis@funai.edu.ng}
\and
McSylvester  Omaba$^{2}$\\
{\tt $^{2}$Email:  mcomaba@uhb.edu.sa}}
\date{$^1$Department of Mathematics/Computer Science/Statistics/Informatics, 
Alex Ekwueme Federal University, Ndufu-Alike, Nigeria.\\[0.3cm]
$^2$Department of Mathematics, College of Science, 
University of Hafr Al Batin,  P. O Box 1803 Hafr Al Batin 31991, KSA.\\[0.3cm]
}

\maketitle

\begin{abstract}
Filtration of quantifiable objects by smoothing kernels on Riemannian manifolds 
for visualisation is an ongoing research. However, using common filters created 
for linear domains on manifolds with non-Euclidean topologies can yield misleading 
results. While there is a lot of ongoing research on convolution of quantifiable 
functions with  smoothing kernels on the lower dimensional manifolds, higher-dimensional 
problems particularly pose a challenge. One important generalization of lower 
dimensional compact Riemannian manifolds is the unit hypersphere. In this paper, 
we derive explicit forms of convolution formulae for Gegenbauer kernel filtration 
on the surface of unit hypersphere. We prove that the Gegebauer filtration is the 
limit of a sequence of finite linear combinations of the hyperspherical Legendre harmonics, 
among other results.
\end{abstract}
\textbf{\emph{Keywords:}} Spherical Laplacian; hypersphere; 
Gegenbauer kernel; convolution; rotations; special orthogonal group $SO(n).$\\
\textbf{\emph{2000 MSC Class:}} 11M06;  11M35; 11M36;  35P05.
 
\section{Introduction}
Spherical harmonic analysis is basically the spectral theory of a differential operator such as  
the spherical Laplacian $\Delta _{S^{n}}$ on a compact Riemannian manifold e.g. 
the hypersphere,  $(S^{n}, g).$ In this analysis, spherical harmonics play salient roles. 
Spherical harmonics analysis is a process of decomposing a function on a sphere into 
components of various wavelengths using surface spherical harmonics as base functions, \cite{IBV}. 
Spherical synthesis combines components of various wavelengths to generate function values
on a sphere and is the reverse process of harmonic analysis. Spherical harmonic analysis 
and synthesis have many real-life and scientific applications, see e.g. \cite{IBV} and \cite{TBU}.
 
Harmonics analysis on Euclidean spaces and in particular the theory of 
spherical harmonics provide an elegant and instructive application of group 
theoretical concepts to various questions of classical function theory, see e.g. \cite{Agata}. 
Eigenspaces of spherical Laplacian with irreducible components with respect to the natural action 
of the group $SO(n)$ within the $L^2$ space on the sphere is an interesting occurrence. 
The role of classical orthogonal polynomials such as the Gegenbauer polynomials  as reproducing 
kernels for the spaces of spherical harmonics of a given degree, or more generally, 
as providing an explicit construction of symmetry adapted basis functions for those 
spaces has been studied extensively, see e.g. \cite{RCP2} and \cite{OM}. 
Also \cite{STRA} studied the connection of the Fourier transform on the Euclidean space to the 
Hankel transform obtained via restriction to $SO(n)$-finite functions and various integral 
identities of the Hecke–Bochner type resulting there. 
The generalized concept of convolution on groups is intimately related to the concept of 
filtering on homogeneous spaces. Some insight into spherical filtering with particular 
emphasis on wavelet transform can be found in \cite{Heal},  \cite{Anto} and \cite{IBV}.

The goal of the present paper is to present a novel form of the Gegenbauer kernel filtration 
of harmonic functions on the hypersphere. This puts in limelight signal processing methods 
on non-Euclidean spaces and in particular on the hypersphere.  The most basic is the notion 
of Fourier transform, that on the sphere corresponds to the expansion of functions into series 
of familiar spherical harmonics. Vast amount of literature is available on such expansions, 
mostly from quantum mechanics and mathematical physics, see e.g.  \cite{Sze}. 

It is known from group representation theory that Fourier transform can be defined on 
any compact Lie group and consequently on homogeneous spaces of these groups \cite{KAW}. 
An example is the group of rotations of a three-dimensional space about the origin, $SO(3),$ 
and the $2$-sphere, $S^{2},$ as a corresponding homogeneous space. We derive some general 
formulae for the convolution of $SO(n)$-finite functions, including recent generalizations 
of Fourier Spherical Harmonic expansions and discuss their function theoretic consequences. 
In particular,  we derive the convolutions with Gegenbauer kernel on the hypersphere, 
$S^{n-1} \subset \BR ,$ regarded as a homogeneous space of the group of rotations of the 
$n$-dimensional space, $SO(n).$  We proceed with clarifying basic terms and concepts. 

\section{Notations and preliminaries}
By a hypersphere in $\BR$ we refer to a set of points whose
Euclidean distance from the origin is equal to unity, that is
\begin{equation}\label{id1}
S^{n-1} = \{ x \in \BR : x^{T}x = 1 \} 
\end{equation}
where $n > 3.$ When $n=2,$ $S^{1}$ is the unit circle and for $n=3$ we have the unit 
$2$-sphere, $S^{2}.$ The superscript $n-1$ in $S^{n-1}$ 
is due to the fact that a sphere in $\BR$ as defined by 
(\ref{id1}) is a submanifold of dimension $n-1;$ see e.g. \cite{JOST} and \cite{Lee} for details.  
The hypersphere $S^{n-1}$ may be parameterized by a set of hyperspherical polar coordinates. 
If $(x_{1} , x_{2}, \cdots , x_{n})$ are Cartesian
coordinates in $\BR ,$ then we define the angles $\theta _{1} , \theta _{2} , \cdots 
\theta _{n-1}$ with $\theta _{1} ,\theta _{2} , \cdots , \theta _{n-2} \in [0,\pi]$  and 
$\theta _{n-1} \in [0,2 \pi]$  such that 
\begin{eqnarray}\label{id2}
\left.\begin{array}{rcl}
x_{1} &=&  \cos \theta _{1}\\
x_{2} &=&  \sin \theta _{1} \cos \theta _{2} \\
x_{3} &=&  \sin \theta _{1} \sin \theta _{2} \cos \theta _{3} \\
x_{4} &=&  \sin \theta _{1} \sin \theta _{2} \sin \theta _{3} \\
\vdots \\
x_{n-1} &=&  \sin \theta _{1} \sin \theta _{2} \cdots \sin \theta _{n-2}\cos \theta _{n-1} \\
x_{n} &=&  \sin \theta _{1} \sin \theta _{2} \cdots \sin \theta _{n-2}\sin \theta _{n-1} .
\end{array}\right\} 
\end{eqnarray}
This choice of coordinates is not unique, but it is a natural generalization of 
spherical polar coordinates in $\RR ^{3} .$  In the familiar case of 
$S^{2} \subset \RR ^{3},$  $\theta_{1}$ 
corresponds to the elevation and $\theta _{2}$ corresponds to the azimuth, 
often denoted $(\theta , \psi) = (\theta _{2} , \theta _{1}).$ 

Basic concepts related to function theory on the hypersphere used throughout this paper 
include the Euler Gamma function, volume element, surface element and some functional spaces 
over the hypersphere. We briefly explain them here. As it is well known, for any $x \in \RR ,$ 
we denote the smallest integer that is greater or equal to $x$ by $[x].$ Also, the 
Euler Gamma function, $\Gamma ,$ in argument $s \in \CC $ is 
\[\displaystyle{\Gamma (s) := \int _{0}^{\infty} e^{-t} t^{s-1} dt , ~~ \Re (s) > 0 }. \]
It can be shown that $\Gamma (\frac{1}{2}) = \sqrt{\pi}$ and 
$\displaystyle{\Gamma (n + \frac{1}{2}) = \frac{(2n)}{2^{2n}n!}\sqrt{\pi} } .$ 
On the hypersphere $S^{n},$ we use  $dV_{n}$ for $n$-dimensional volume element and
 $dS_{n}$ for $n$-dimensional surface element over the unit sphere $S^{n}.$ 
 We also use $d \xi$ for surface element over a general domain. 

A point on the hypersphere is specified as follows. For $n \geq 3 ,$  let 
$\xi = \xi _{(n)} \in S^{n}$ then  $\xi _{(n)} = t e_{n} + \sqrt{1-t^{2}}\xi _{(n-1)}$ 
with $t \in [-1,1] ,~  \xi _{(n-1)} \in S^{n-1}$ and 
$ e_{n} = (0, 0, \cdots , 0, 1)^{T}$ is basis vector in $\BR .$  When we write 
$\xi _{(n-1)} \in S^{n}$ we mean that 
$\xi = (\xi _{1} , \xi _{2} , \cdots , \xi _{n-1} , 0 ) \in S^{n}.$ 

It can be shown that for $n \geq 3,$
\begin{equation}\label{surf1}
\Big(t e_{n} + \sqrt{1-t^{2}}\xi _{(\xi _{(n-1)})} \Big) d S_{n-1} = 
 (1-t^{2})^{\frac{n-3}{2}} dt dS_{n-2} (\xi _{(n-1)}) 
 = (1-t^{2})^{\frac{n-3}{2}} dt dS_{n-2}  . 
\end{equation}
 
The surface area of the hypersphere is by definition given by 
\[ |S^{n-1}| := \int _{S^{n-1}} d S_{n-1} \] 
and so from (\ref{surf1})  we have  
\[ |S^{n-1}| = \int _{-1}^{1} (1 -t^{2})^{\frac{n-3}{2}} dt \int _{S^{n-2}} d S_{n-2}  
= |S^{n-2}| \int _{-1}^{1} (1 -t^{2})^{\frac{n-3}{2}} dt .\]
In the integral, 
$\displaystyle{\int _{-1}^{1} (1 -t^{2})^{\frac{n-3}{2}} dt} ,$ let  $s=t^{2}$ so that 
\[ \int _{-1}^{1} (1 -t^{2})^{\frac{n-3}{2}} dt = \int _{0}^{1} s^{-1/2} (1-s)^{\frac{n-3}{2}} 
ds  = \beta (\frac{1}{2} , \frac{n-1}{2})  = 
\frac{\sqrt{\pi} \Gamma (\frac{n-1}{2})}{\Gamma (\frac{n}{2})} .\]
Thus we have the recursive relation 
 \begin{equation}\label{surf2}
 |S^{n-1}| = \frac{\sqrt{\pi} \Gamma (\frac{n-1}{2})}{\Gamma (\frac{n}{2})}|S^{n-2}| 
 = \frac{2 \pi ^{n/2}}{\Gamma (n/2)}; ~~ n \geq 3 .
 \end{equation}

If $A \in \RR ^{n \times n}$ is orthogonal, then there exists $\xi \in S^{n-1}$ and 
$x \in \BR$ such that 
\[ d S_{n-1} (A \xi)  =  d S_{n-1} (\xi) ~~ \text{and} ~~ 
d V_{n} (A x)  =  d V_{n} (x) . \]
In polar coordinates, we relate $x_{(n)}$ with $\xi _{(n)}$ using 
\[ x_{(n)} = |x| \xi _{(n)},  ~~
\xi _{(n)} = \xi  \in S^{n-1} . \]
Thus, 
\begin{equation}\label{vol}
dV_{n}(|x| \xi) = |x|^{n-1} d |x| dS_{n-1}(\xi ) .
\end{equation}
On the hypersphere,  $(S^{n},g),$ where $g$ is a Riemannian metric,  
we define the Riemannian (volume) measure $dV _{g}$ in terms of the metric $g$ 
in standard hyperspherical coordinates (\ref{id2})  as 
\begin{equation}\label{col10}
dV _{g} = \sin ^{n-1} d \theta d \omega := 
\sin ^{n-1} d \theta \sin ^{n-2}\sin ^{n-3} \cdots 
\sin \theta_{2}d \theta_{2} \cdots d \theta_{n-1} .
\end{equation}

The space of real or complex-valued continuous functions $C(S^{n-1})$ over the unit 
hypersphere $S^{n-1}$ is a Banach space with the canonical norm 
\[ ||f||_{\infty} = \sup _{\xi \in S^{n-1}} |f(\xi )|  .\]
Similarly, the space of real or complex-valued 
square integrable functions $L^{2}(S^{n-1})$ over  $S^{n-1}$ is an Hilbert space 
with canonical inner product 
\begin{equation}\label{ips}
\lka f ,g \rka := \int _{S^{n-1}}f \bar{g} d S_{n-1}
\end{equation}
with the induced norm 
\begin{equation}\label{ipn}
|| f ||_{2} := \sqrt{\lka f ,f \rka} .
\end{equation}
The space $L^{2}(S^{n-1})$ is the closure of $C(S^{n-1})$ with respect to the norm 
(\ref{ipn}). We study a function $f$ on $S^{n-1}$ through its extension 
\[ f^{\ast} (x) := f (\frac{x}{|x|}) = f (\xi ) ~~\text{for} ~~ 
x \in \Omega _{\epsilon} := \{ x \in \BR : |x| \in [1- \epsilon , 1 + \epsilon ] \} . \]

\section{Gegenbauer Polynomial}
Gegenbauer polynomials enables one express Gegenbauer filtration kernels. Thus, 
following \cite{Avery, KAW, MWW, OM} and \cite{MMT}, we briefly explore them in this section. 

The system of Gegenbauer polynomials, which we denote by $C_{l}^{\alpha}$  
for degree $l$ and index $\alpha ,$ are ultraspherical polynomials expressible  
through hyperspherical Legendre polynomials $\{ P_{l,n}(x): l = 0,1,2,\cdots \}.$  
Hyperspherical Legendre polynomials are a classical family of polynomials orthogonal with 
respect to the weight function $w _{\alpha} = (1-x^{2})^{\alpha-\frac{1}{2}}$ on 
the support interval $[-1,1] .$  We make these definitions formal in what follows.

\begin{definition}
The function $P_{l,n}$ is hyperspherical Legendre polynomial of degree $l$ in $n$ dimension 
and it is given by 
\begin{equation}\label{Lpoly}
P_{l,n} (t) := l! \Gamma (\frac{n-1}{2}) \sum _{j=0}^{[\frac{l}{2}]} (-1)^{j} 
\frac{(1-t^{2})^{j}t^{l-2j}}{4^{j}j! (l-2j)! \Gamma (j+ \frac{n-1}{2})}
\end{equation}
for $t \in [-1,1].$
\end{definition}
Specifically, $P_{l,n}(1) = 1,$ for example. 

There are other forms of definition of  $P_{l,n}.$ Popular among those definitions 
is the Rodrigues formulation.
\begin{theorem}\label{Rodr}(Rodrigues representation). 
The hyperspherical Legendre polynomial of degree $l$ in $n$ dimension \eqref{Lpoly}  
in Rodrigues form  is given by 
\begin{equation}\label{Rodrf}
P_{l,n}(t) = (-1)^{l} R_{l,n}
 (1 - t^{2})^{\frac{3 - n}{2}}\frac{d^{l}}{dt^{l}}(1 - t^{2})^{l + \frac{n - 3 }{2}}
~~\mathrm{with}~~ n \geq 2,
\end{equation}
where the Rodrigues constant $R_{l,n}$ is given by
 \begin{equation}\label{Rodrc}
R_{l,n} = \frac{\Gamma(\frac{n -1}{2})}{2 ^{l} \Gamma(l + \frac{n - 1}{2})} . 
\end{equation}
\end{theorem}
For proof, see e.g. \cite{MMT, KAW} and \cite{OM14}.

We remark that for $n=2,$ we have the relation 
\[ \Gamma(l + \frac{1}{2})= \frac{(2l)!}{2^{2l}l!}\Gamma(\frac{1}{2})\]
and so 
\[ P_{l,2}(t)= (-1)^{l} \frac{2^{l}l!}{(2l)!} (1-t^{2})^{1/2}\Big( \frac{d}{dt}\Big)^{l}
(1-t^{2})^{l-1/2}, ~~~ l \in \mathbb{N}_{0}\]
which can be simplified to 
\[ P_{l,2}(t)= \cos (l \arccos t); ~~ t \in [-1,1] .\]
For $n=3,$ one recovers from \eqref{Rodrf} the standard Rodrigues representation 
formula for the standard Legendre polynomials as 
\begin{equation}\label{Lep1}
P_{l,3}(t)= \frac{1}{2^{l}l!}\Big( \frac{d}{dt}  \Big)^{l}(t^{2}-1)^{l}, ~~~ l \in \mathbb{N}_{0}.
\end{equation} 
Moreover, \cite{MMT} proved an integral representation of the hyperspherical harmonics 
as the theorem \eqref{hsi} that follows.
\begin{theorem}\label{hsi}
For $l \in \mathbb{N}_{0},$ $n\geq 3$ and for any $t \in [-1,1] ,$ 
\begin{equation}\label{hsi1}
P_{l,n}(t) = \frac{|S^{n-3}|}{|S^{n-2}|}\int_{-1}^{1}\Big[ t + 
i (1-t^{2})^{\frac{1}{2}}) s \Big]^{l} (1-s^{2})^{\frac{n-4}{2}} ds.
\end{equation}
\end{theorem}

Now, let $\mathcal{H}_{l,n}$ denote the space homogeneous Legendre polynomials of degree $l$ 
in dimension $n.$ We call function $f \in \mathcal{H}_{l,n}$ such that $\Delta f =0$ 
hyperspherical  harmonic.  We denote the space of hyperspherical harmonic polynomials 
restricted to the unit hypersphere, $S^{n},$ by $\mathcal{Y}_{l,n}.$  
So any $Y_{l} \in \mathcal{Y}_{l,n}$ is related to a homogeneous 
harmonic $ h_{l} \in \mathcal{H}_{l,n}$ as follows $h_{l} (r \xi) = r^{l}Y_{l}(\xi )$ 
where $r =  |h_{l}|.$ So, they have the same dimension. 

The hyperspherical harmonic polynomials have the following properties.

\begin{theorem} (Addition theorem). \label{addthm} 

Let $\{ \psi _{l,j}: 1 \leq j \leq  d_{l}(n)  \}$ be an orthonormal basis of the
 space of  $n$-dimensional spherical harmonics  $\mathcal{H}_{l}(S^{n}),$ i.e:
\begin{equation}\label{gph10a}
\int _{S^{n}}\psi _{l,j}(x) \bar{\psi }_{l,m}(x) dV_{g}(x) = \delta _{jm};~~1 
\leq j,m \leq d_{l}(n).
\end{equation}
Then
\begin{equation}\label{gph11}
\sum _{j = 1}^{d_{l}(n)} \psi _{l,j}(x) \bar{\psi }_{l,m}(y) = 
\frac{d_{l}(n)}{|S^{n}|} P_{l,\frac{(n - 1)}{2}}(x \cdot y)
\end{equation}
where $P_{l,n}(t)$ are the hyperspherical Legendre polynomials of degree  
$l$ in $n$ dimensions. The dimension of $\mathcal{H}_{l}(S^{n})$ is $d_{l}(n)$ 
given by the formula 
\[ d_{l} (n) = 
\begin{pmatrix}
l + n \\
n
\end{pmatrix}
-
\begin{pmatrix}
l + n - 2 \\
n
\end{pmatrix}   = \frac{(2l + n - 1)(l + n - 2)!}{l! (n - 1)!},
 ~~  \text{for} ~~  l \in  \mathbb{N} .\]
\end{theorem}
For proof, one may see \cite{OM} and \cite{MMT}. This means in particular 
that $P_{l,(n-1)/2}(x \cdot y)$ is a harmonic function on $S^{n}$  with eigenvalue 
$\lambda_l = l(l+n-1)$ for the  eigenvalue problem  
\[\DEL _{S^{n}} P_{l,(n-1)/2}(\theta)  = \lambda_l  P_{l,(n-1)/2}(\theta)\]
where  $\DEL _{S^{n}}$ is the usual Laplacian on $S^{n}$ defined by 
\begin{equation}\label{gphe}
\DEL _{S^{n}} :=  \frac{1}{\sin ^{n -1} \theta } \frac{\partial }{\partial \theta}
\big( \sin ^{n -1}   \frac{\partial  }{\partial \theta } \big )  +
 \frac{1}{\sin ^{2} \theta}
 \DEL _{S^{n-1}}  .
 \end{equation}

We note that \eqref{gph11} can also be expressed in terms of Jacobian 
polynomials $P_{l}^{\alpha, \alpha}$  as follows:
\begin{equation}\label{jac1}
\sum _{j = 1}^{d_{l}(n)} \psi _{l,j}(x) \bar{\psi }_{l,m}(y) = \frac{1}
{P_{l}^{\big(\frac{(n - 2)}{2}, \frac{(n - 2)}{2}\big)}(1)} P_{l}^{(\frac{(n - 2)}
{2}, \frac{(n - 2)}{2})}(\cos \theta) ~~ \text{for} ~~ x \cdot y = \cos \theta .
\end{equation}

A special case of the hyperspherical Legendre polynomials are the familiar spherical harmonics
on $S^{2}$ usually given by
\begin{equation}\label{id43}
Y_{l,m}(\theta , \psi) = \tilde{N}_{lm} P_{l,m} (\cos \theta )\exp (j m \psi)
\end{equation}
with normalisation constant 
\begin{equation}\label{id44}
\tilde{N}_{lm}  =   \sqrt{(2l + 1)\frac{(l -m)!}{(l+m)!}}
\end{equation}
chosen so that the norm of each harmonic would be unity.

We have the following lemma. 
\begin{lemma}\label{addt}
The hyperspherical harmonic polynomials are orthogonal:
\begin{eqnarray}\label{ifa}
\int _{S^{n}} P_{l,n}(x \cdot y)P_{j,n}(x \cdot y) 
dV_{n}(x) &= &
\begin{cases} 
 \frac{|S^{n}|}{d_{l}(n)} & \text{if} ~~ l = j . \\
 0 & \text{if} ~~ l \neq j .
 \end{cases}
\end{eqnarray}
\end{lemma}

\begin{proof}
That 
\[\int _{S^{n}} P_{l,n}(x \cdot y)P_{j,n}(x \cdot y) dV_{n}(x) =0 ~~ \text{for} ~~ l \neq j \]
follows directly from \eqref{gph10a}.
When $l=j$ we observe from  (\ref{gph11}) that 
\begin{eqnarray*}
\int _{S^{n}} |P_{l,n}(x \cdot y)|^{2} dV_{g}(y) & = &
\big(\frac{|S^{n}|}{d_{l}(n)} \big)^{2} \int _{S^{n}} \big| \sum _{j = 1}^{d_{l}(n)} 
\psi _{l,j}(x) \bar{\psi }_{l,m}(y) \big|^{2}  dV_{g}(y) \\
&=& \big(\frac{|S^{n}|}{d_{l}(n)} \big)^{2}  \sum _{j = 1}^{d_{l}(n)} 
\big| \psi _{l,j}(x)  \big|^{2} 
= \big(\frac{|S^{n}|}{d_{l}(n)} \big)^{2} \frac{d_{l}(n)}{|S^{n}|} .
\end{eqnarray*}
\end{proof}

We also have 
\begin{equation}\label{Rint2}
\int _{-1}^{1}P_{l,n}(t) P_{j,n}(t)(1-t^{2})^{\frac{n-3}{2}}dt  = 
\frac{|S^{n-1}|}{d_{l}(n)|S^{n-2}|}\delta_{lj} .
\end{equation}

Besides, that 
\[\int _{S^{n-1}} P_{l,n}( \xi \cdot \eta )  P_{j,n}( \xi \cdot \eta ) dS_{n-1} (\xi) = 
\left\{\begin{array}{rcl}
0 ~~~ &=& \text{if} ~~ l \neq j \\
\frac{|S^{n-1}|}{d_{l}(n)} &=& \text{if} ~~ l = j
\end{array}\right. \]
follows from the following. Take $t = \xi \cdot \eta$ and observe that 
\[\int _{-1}^{1} |P_{l,n}(t)|^{2} (1-t^{2})^{\frac{n-3}{2}} dt 
= \frac{|S^{n-1}|}{d_{l}(n) |S^{n-2}|} .\]
So, 
\[\int _{S^{n-1}}\Big( \int _{-1}^{1} P_{l,n}(t) P_{j,n}(t) 
(1-t^{2})^{\frac{n-3}{2}} dt\Big) dS_{n-2} = |S^{n-2}|
\int _{-1}^{1} P_{l,n}(t) P_{j,n}(t) 
(1-t^{2})^{\frac{n-3}{2}} dt \]
where
\[\int _{-1}^{1} P_{l,n}(t) P_{j,n}(t) 
(1-t^{2})^{\frac{n-3}{2}} dt  = 0  ~~ \text{for}~~ l \neq j .\]

Another interesting  property is the projection of integrable function onto the space 
of spherical harmonics on the hypersphere. We make the following definition. 
\begin{definition}\label{defproj}
The projection of $f \in L^{1}(S^{n-1})$ into  $\mathcal{Y}_{l,n}$ is 
\begin{equation}\label{approx4}
(F_{l,n}f)(\xi ) := \frac{d_{l}(n)}{|S^{n-1}|} \int _{S^{n-1}}   P_{l,n} (\xi \cdot \eta)
f(\eta)dS_{n-1}(\eta), ~~ \eta \in S^{n-1}.
\end{equation}
\end{definition}

Let us now consider the problem of finding the best approximation for 
$f \in L^{2} (S^{n-1}).$  That is, 
\begin{equation}\label{approx2}
\inf _{Y_{n} \in \mathcal{Y}_{l,n}} || f - Y_{l} ||_{L^{2}(S^{n-1})} .
\end{equation}
In terms of the orthonormal $\{ Y_{l,j} : 1 \leq j \leq  d_{l}(n) \}$ 
of  $\mathcal{Y}_{l,n},$ the solution to (\ref{approx2}) is 
\begin{equation}\label{approx3}
(F_{l,n}f)(\xi ) = \sum _{j=1}^{d_{l}(n)}  \lka f , Y_{l,j}  \rka Y_{l,j}(\xi ) .
\end{equation}
The operator $F_{l,n}$ is the projection of any $f$ into $\mathcal{Y}_{l,n}$ and it is defined 
for $f \in L^{1}(S^{n-1}).$ One can apply this formula without explicit knowledge 
of the orthonormal basis by applying the addition lemma \eqref{addt}.

\begin{lemma}
If $f \in L^{2}(S^{n-1})$  then for any $\xi \in S^{n-1}$ we have
\[ || (F_{l,n}f) ||_{L^{2}(S^{n-1})} \leq ||f||_{L^{2}(S^{n-1})} .\]
\end{lemma}

\begin{proof}
Let $f \in L^{2}(S^{n-1})$  and  $\xi \in S^{n-1}$ given,  we have 
\begin{eqnarray*}
|(F_{l,n}f)(\xi )|^{2} &\leq & \frac{d_{l}(n)}{|S^{n-1}|} \int _{S^{n-1}}  
| P_{l,n} (\xi \cdot \eta)|^{2} d S_{n-1}(\eta)\cdot \int _{S^{n-1}}
|f(\eta)|^{2}dS_{n-1}(\eta)  \\
\implies || (F_{l,n}f) ||_{L^{2}(S^{n-1})} &\leq & \sqrt{d_{l}(n)} ||f||_{L^{2}(S^{n-1})} .
\end{eqnarray*}
Similarly, 
\[ || (F_{l,n}f) ||_{L^{2}(S^{n-1})} \leq  \sqrt{\frac{d_{l}(n)}{|S^{n-1}|}} 
||f||_{L^{2}(S^{n-1})} .\]
These imply that 
\[ || (F_{l,n}f) ||_{L^{2}(S^{n-1})} \leq ||f||_{L^{2}(S^{n-1})} .\]
\end{proof}

All these give rise to the following theorem. 
\begin{theorem}
We have the orthogonal decomposition 
\[ L^{2}(S^{n-1}) = \bigoplus _{l=0}^{\infty} \mathcal{Y}_{l,n} . \]
\end{theorem}
This means that any $f \in L^{2}(S^{n-1})$ can be uniquely represented as 
\begin{equation}\label{dec1}
f(\xi)= \sum _{l=0}^{\infty}f_{l}(\xi) ; ~~ \text{with}~~ f_{l} \in \mathcal{Y}_{l,n}; ~ l \geq 0.
\end{equation}
We call  $f_{l} \in \mathcal{Y}_{l,n}$ hyperspherical component of $f$ given by 
\begin{equation}\label{dec2}
f_{l}(\xi)= \frac{d_{l}(n)}{|S^{n-1}|} \int _{S^{n-1}} 
f(\eta)P_{l,n}(\xi \cdot \eta)dS_{n-1}(\eta);  ~~~ \eta \geq 0.
\end{equation}
We now define another generalisation of the hyperspherical harmonics. 
\begin{definition}
For $\alpha > 0 ,$ $l \in \mathbb{N}_{0}$, the Gegenbauer polynomial of degree $l$ with index 
$\alpha$ is defined by 
\begin{equation}\label{gege1}
C_{l}^{\alpha} (t):= \binom{l+2 \alpha -1}{l} 
\frac{\Gamma (\alpha + 1/2)}{\sqrt{\pi}\Gamma (\alpha)} 
\int _{-1}^{1}(t + i(1-t^{2})^{l/2} s)^{l} (1-s^{2})^{\alpha -1} ds .
\end{equation}
\end{definition}
\begin{theorem}
The Gegenbauer function $C_{l}^{\alpha}$ is indeed a polynomial and has a representation in 
terms of the hyperspherical harmonics as 
\begin{equation}\label{gege2}
C_{l}^{\frac{n-2}{2}} (t):= \binom{l+n -3}{l} P_{l,n}(t);  ~~~ \text{for}~~~ n \geq 3.
\end{equation}
\end{theorem}
\begin{proof}
From the definition of $C_{l}^{\alpha}$ in \eqref{gege1}, we observe that by Binomial expansion 
\[ [t+ i(1-t^{2})^{\frac{1}{2}}s]^{l}  = \sum _{j=0}^{l}\binom{l}{j} 
t^{l-j}(1-t^{2})^{\frac{j}{2}}(is)^{j} .\]
So for $j=2k+1$ odd, the integral corresponding term is 
\[ \int _{-1}^{1} s^{2k+1}(1-s^{2})^{\alpha -1} ds = 0. \]
Thus, $C_{l}^{\alpha}$ is real valued and has the form 
\[ C_{l}^{\alpha} (t)= \binom{l+2 \alpha -1}{l} \frac{\Gamma (\alpha +1/2)}{\sqrt{\pi}
\Gamma (\alpha)}
\sum _{k=0}^{[\frac{l}{2}]}\binom{l}{2k} 
t^{l-2k} (-1)^{k}(1-t^{2})^{k} \int _{-1}^{1}s^{2k}(1-s^{2})^{\alpha-1} ds \]
which is a polynomial of degree less or equal to $l.$ 

Besides, the coefficient of $t^{l}$ in $C_{l}^{\alpha}(t)$ is 
\[\binom{l+2 \alpha -1}{l} \frac{\Gamma (\alpha +1/2)}{\sqrt{\pi}
\Gamma (\alpha)} \sum _{k=0}^{[\frac{l}{2}]}\binom{l}{2k} 
 \int _{-1}^{1}s^{2k}(1-s^{2})^{\alpha-1} ds .\]
 Therefore, $C_{l}^{\alpha}(t)$ is a polynomial.
 
Moreover, from \eqref{hsi1}, it follows that $C_{l}^{\alpha}(t)$ has the form \eqref{gege1}.
\end{proof}

We remark the following standard result. 
\begin{lemma}\label{gph70}
The Gegenbauer polynomials $C_{l}^{(n-1)/2} (\xi )$ are eigenfunctions of the Laplace-Beltrami 
operator on $S^{n-1}$ which satisfy the following eigenvalue problem: 
\begin{equation}\label{gph7}
(1-\xi ^{2})(C_{l}^{(n-1)/2})^{\prime \prime}(\xi ) - n \xi (C_{l}^{(n-1)/2})^{\prime} (\xi )
 + l(l + n - 1)C_{l}^{(n-1)/2}(\xi ) = 0.
\end{equation}
\end{lemma}

\begin{proof}
Let $\psi \in  C^{\infty}(S^{n}).$  Then,  
\[\DEL _{S^{n}} \psi  =  \frac{1}{\sin ^{n -1} \theta } \frac{\partial }{\partial \theta}
\big( \sin ^{n -1}   \frac{\partial \psi }{\partial \theta } \big )  +
 \frac{1}{\sin ^{2} \theta}
 \DEL _{S^{n-1}} \psi . \]
For a fixed $\xi  \in S^{n},$ the function $\psi$ given by  
$C_{l}^{n}(\xi  \cdot \eta) = C_{l}^{{(n-1)/2}}(\cos \theta )$ belongs to 
$\mathcal{H}_{l}(S^{n})$ where $\theta$ is the geodesic angle between 
$\xi $ and $\eta \in S^{n} .$ 
A direct computation gives  
\begin{equation*}
 \DEL _{S^{n}} \psi   = l(l + n -1)\psi .
 \end{equation*}

Since $\psi$ only depends on $\xi ,$ so,
$\DEL _{S^{n-1}} \psi = 0 $  and thus 
\begin{equation}\label{col9}
 l(l + n- 1)C_{l}^{n}(\xi ) = \DEL _{S^{n}}C_{l}^{n}(\xi ) = 
(1 - \xi ^{2})\frac{\partial ^{2}C_{l}^{n}(\xi ) }{ \partial t^{2}} 
- n \xi \frac{\partial C_{l}^{n}(\xi ) }{ \partial \xi } 
\end{equation}
proving the theorem.
\end{proof}

\section{Spherical harmonic analysis and filtration}
Consider the following generating function of the Gegenbauer harmonic \eqref{gege1} for a fixed 
index $\alpha$:
\begin{equation}\label{n1}
\Phi (r) = \sum _{l=0}^{\infty} C_{l}^{\alpha}(t) r^{l}= \frac{1}{(1+r^{2}-2rt)^{(\frac{n-2}{2})}}
 ~~ ~~ \text{for} ~~ |r|<1,~~  |t| \leq 1.
\end{equation}
Consequently from \eqref{gege2}, 
\begin{equation}\label{n2}
\Phi (r) = \sum _{l=0}^{\infty}\binom{l+n-3}{n-3} P_{l,n}(t) r^{l}
= \frac{1}{(1+r^{2}-2rt)^{(\frac{n-2}{2})}}
 ~~ ~~ \text{for} ~~ |r|<1, ~~ |t| \leq 1.
\end{equation}
\begin{definition}
A Gegenbauer filtration kernel is  
\begin{equation}\label{n3}
G_{n} (r,t) := \frac{|S^{n-2}|}{|S^{n-1}|}\frac{1-r^{2}}{(1+r^{2}-2rt)^{(\frac{n}{2})}}
 ~~ ~~ \text{for}~~ |r|<1,~~  |t| \leq 1 ~~~ \text{and} ~~~ n \geq 3.
\end{equation}
\end{definition}
\begin{lemma}\label{filt1}
The function \eqref{n3} satisfies the following properties 
\begin{equation}\label{n4}
\int_{-1}^{1}G_{n} (r,t) (1-t^{2})^{\frac{n-3}{2}} dt = 1
\end{equation}
and 
\begin{equation}\label{n5}
\lim _{r \rightarrow 1^{-}} G_{n} (r,t) =0.
\end{equation}
\end{lemma}
\begin{proof}
A direct computation shows that 

\begin{eqnarray*}
\int_{-1}^{1}G_{n} (r,t) (1-t^{2})^{\frac{n-3}{2}} dt 
&=& \frac{|S^{n-2}|}{|S^{n-1}|} \int_{-1}^{1}   
\sum _{l=0}^{\infty} d_{l}(n) r^{l} P_{l,n}(t) (1-t^{2})^{\frac{n-3}{2}} dt \\ 
&=& \frac{|S^{n-2}|}{|S^{n-1}|} \int_{-1}^{1}   
 (1-t^{2})^{\frac{n-3}{2}} dt  =1 
\end{eqnarray*}
showing \eqref{n4}. 

Besides for $n\geq 2$ we have 
\begin{eqnarray*}
\lim _{r \rightarrow 1^{-}} \Big(\frac{1-r^{2}}{(1+r^{2}-2rt)^{\tfrac{n}{2}}} \Big)
&=& \lim _{r \rightarrow 1^{-}} \Big(\frac{1-r^{2}}{[(1-r)^{2}+2r(1-t)]^{\tfrac{n}{2}}} \Big) \\
&\leq & \lim _{r \rightarrow 1^{-}} \Big(\frac{1-r^{2}}{[2r(1-t_{0})]^{\tfrac{n}{2}}} \Big)
=0 
\end{eqnarray*}
for $t \in [-1,t_{0}]$ arbitrary, showing   \eqref{n5}.
\end{proof}

Now we define the Gegenbauer filtration operator $\mathcal{G}_{n}(r)$ to be 
\begin{equation}\label{filtdef}
(\mathcal{G}_{n}(r)f)(\xi) := \frac{1}{|S^{n-2}|}\int _{S^{n-1}}G_{n}(r , \xi \cdot \eta)
f(\eta) dS_{n-1}.
\end{equation}
We immediately have the following result.
\begin{theorem}\label{main1}
The Gegebauer filtration $\mathcal{G}_{n}(r)$ is the limit of a sequence of finite linear 
combinations of the hyperspherical Legendre harmonics.  
\end{theorem}
\begin{proof}
Let $f \in L^{2}(S^{n-1}).$ From \eqref{filtdef},   we have 
\begin{eqnarray*}
(\mathcal{G}_{n}(r)f)(\xi) &=& \frac{1}{|S^{n-2}|}\int _{S^{n-1}}G_{n}(r , \xi \cdot \eta)
f(\eta) dS_{n-1} \\
&=& \frac{1}{|S^{n-2}|}\int _{S^{n-1}}
\Big( \int_{-1}^{1}  G_{n}(r , t) P_{l,n}(t)(1-t^{2})^{\tfrac{n-3}{2}} dt \Big)
f(\eta) dS_{n-1}\Big|_{t= \xi \cdot \eta} \\
\implies (\mathcal{G}_{n}(r)f)(\xi) &=& \frac{1}{|S^{n-2}|} 
\sum _{l=0}^{\infty} r^{l} d_{l}(n)\int _{S^{n-2}} C_{l}^{\alpha} (\xi \cdot \eta)f(\eta)
 dS_{n-1} (\eta) \\
 &=& \sum _{l=0}^{\infty}r^{l}(F_{l,n}f)(\xi )
\end{eqnarray*}
and for $|r|<1,$ 
 \begin{equation}\label{filtdef2}
(\mathcal{G}_{n}(r)f)(\xi) = \sum _{l=0}^{\infty}d_{l}(r) 
P_{l,n}(\xi \cdot \eta)r^{l}(F_{l,n}f)(\xi ) 
= (F_{l,n}f)(\xi )
\end{equation}
which completes the proof.
\end{proof}

The inner product of two complex valued functions on the surface of the hypersphere is 
\begin{equation}\label{m1}
\lka f , h \rka _{S^{n-1}} = \int _{S^{n-1}} f(\xi , \eta)h(\xi , \eta)^{\ast}dS_{n-1}(\eta)
\end{equation}
where $\ast$ denotes complex conjugation. Using the fact that the Gegenbauer filtration 
is rotation invariant over $SO(n),$ one can move $h$ to any point 
$(\xi _{0} , \eta _{0}) \in S^{n-1}.$ Then we define a generalised convolution 
as a function on the rotation group $SO(n)$ to be 
\begin{equation}\label{m2}
(h \ast f)(R) = \int _{S^{n-1}} f(\xi , \eta) h_{R}(\xi , \eta)^{\ast}dS_{n-1}(\eta)
\end{equation}
for $R \in SO(n)$ and where $h_{R}$ is $h$ rotated by $R$ defined as 
\begin{equation}\label{m3}
h_{R}(A) = h(R^{-1}A).
\end{equation}
Thus, every well-behaved function $f \in L^{2}(S^{n-1})$ admits the expansion 
\begin{equation}\label{m4}
f(\xi , \eta) = \sum _{\alpha =0}^{\infty} \sum _{l,m} \hat{f}_{l,m}(\xi , \eta ) 
C_{l}^{\alpha} (\xi , \eta ) 
\end{equation}
where the generalised Fourier transform on the surface of the hypersphere 
$\hat{f}_{l,m}$ is  defined to be 
\begin{equation}\label{m5}
\hat{f}_{l,m}(\xi)  = \int_{S^{n-1} \ni \eta \in SO(n)} f(R)\rho_{l,m}(R)dS_{n-1}(\eta).
\end{equation}
Here,  $\rho(R)$ is a function on $SO(n)$ containing fixed matrix valued functions called 
the irreducible representations of $SO(n).$ 

As a consequence of \eqref{approx3}, we observe that for a suitable $\psi$, we have  
\begin{eqnarray}\label{csq}
(F_{l,n}\psi)(\xi) &=& \frac{d_{l}(n)}{|S^{n-1}|}\int_{S^{n-1}}P_{l,n}(\xi \cdot \eta) 
\psi(\eta)dS_{n-1}(\eta) \nonumber \\
&=&  \frac{d_{l}(n)}{|S^{n-1}|}  \frac{|S^{n-1}|}{d_{l}(n)}\psi_{l}(\xi) 
= \sum _{j=1}^{d_{l}(n)}\lka \psi_{l}, \psi_{l,j} \rka _{S^{n-1}} \psi_{l,j}(\xi) \nonumber \\
\implies (F_{l,n}\psi)(\xi) &=& \sum _{j=1}^{d_{l}(n)}\hat{\psi}_{l}\psi_{l,j}(\xi).
\end{eqnarray}

Another interesting result of this work is that Gegenbauer filtration on $S^{n-1}$ 
coincides with convolution on $S^{n-1}.$ Particularly, convolution of two 
functions on the hypersphere equals the multiplication of their Fourier coefficients. 
This is the next theorem \eqref{idgm} below. 
\begin{theorem}\label{idgm}
Let $f , h \in L^{2}(S^{n-1})$  and $\mathcal{G}(r)$ the Gegenbauer filtration operator. 
Then the spectrum of the Gegenbauer kernel convolution is given by 
\begin{equation}\label{id11}
\widehat{f\ast h}(l , \alpha) =   
 \sqrt{\frac{\Gamma (l + n-2)(2l + n -2)}{l! \Gamma (n-1)}} 
\hat{f} (l , \alpha ) \hat{h}(l ,0).
\end{equation}
\end{theorem}
A special case of this property has been proved for the special case of convolutions 
on $S^{2}$, see e.g. \cite{Heal, TBU} and \cite{Agata}. We get a proof 
for this generalisation on $S^{n-1}.$ 
\begin{proof}
Since  $f , h \in L^{2}(S^{n-1})$  and by definition of the Gegenbauer filtration, 
\begin{eqnarray*}
\widehat{f\ast h}(l , \alpha) &=&  \int _{S^{n-1}} f\ast h(\xi)  
\overline{\mathcal{G}_{n}(r)(\xi)} d S_{n-1}(\xi )
 \nonumber \\
&=& \int _{S^{n-1}} \Big( \int _{SO(n)} f (R \eta )h (R^{-1} \xi )dR \Big) 
\overline{\mathcal{G}_{n}(r) (\xi)} d S_{n-1}(\xi ).
\end{eqnarray*}
After rearranging the integrals and using the \eqref{approx3} we get 
\[\widehat{f\ast h}(l , \alpha ) = 
\int _{SO(n)} 
\Big( \int _{S^{n-1}} h (R^{-1} \xi )
 \overline{(F_{l,n})(F_{j,n})C_{l}^{\alpha}C_{j}^{\alpha} (\xi)} dS_{n-1}( \xi )
 \Big)f (R N ) dR \]
where $N$  is the north pole. This implies that from \eqref{csq} and the addition theorem 
\eqref{addthm}, 
\[\widehat{f\ast h}(l , \alpha ) = \int _{SO(n)}  f(R N)
\Big( \int _{S^{n-1}} h (R^{-1} \xi )
\sum _{|j| \leq d_{l}(n)} d_{l}(n) C_{j}^{\alpha} (R\xi)F_{j,n}(R^{-1} \xi)dS_{n-1}(\xi)
\Big)C_{l}^{\alpha} (R^{-1}\xi)dR \]
which is $0$ unless $\alpha =0.$ 
This follows since the measure on $SO(n)$ is rotation-invariant. 

Therefore,  
\[\widehat{f\ast h}(l , \alpha )= \int _{SO(n)}  f(R N)C_{l}^{\alpha} 
(R^{-1}\xi)dR \hat{h}(l,0).\]
Finally using the relationship between $C_{l}^{\alpha}$ and the hyperspherical harmonics 
$P_{l,n}$ \eqref{gege2}, we conclude that \eqref{id11} holds. 
\end{proof}

\section{Conclusion}
We have derived formulae for Gegebauer kernel filtration of different 
functions on the unit hypersphere. Particularly, we derived a Gegebauer filtration operator 
via hyperspherical Legendre harmonics. We related the filtration to convolution on the 
hypersphere and proved that Gegenbauer filtration on $S^{n-1}$ 
coincides with convolution on $S^{n-1}.$ We showed that convolution of two 
functions on the hypersphere equals the multiplication of their Fourier coefficients.

The results obtained suggest that filtrations and convolutions on the Euclidean space 
can be successfully extended to spaces with non-Euclidean topologies such as the hypersphere 
using the example of Gegenbauer filters. The formulae and approach are readily 
extendable to other non-Euclidean geometries such as the flat torus and cylinder in higher 
dimensions. These can be done as further extension of this study.


\begin{thebibliography}{99}
\bibitem{Anto} Antoine, J. P.  and Torresani, C. B.,  \textsl{“Wavelets on the 2-sphere: A 
group theoretical approach,”}  Appl. Comput. Harmon. Anal, vol. 7, pp. 1–30, 1999.

\bibitem{KAW}  Atkinson, K. and  Han, W., \textsl{Spherical Harmonics and 
Approximations on the Unit Sphere: An Introduction,}  Springer-Verlag, Berlin, Germany; 2012.

\bibitem{Avery} Avery, J, \textsl{Hyperspherical Harmonics: Applications in Quantum Theory: 
1st edition, ser. Reidel Texts in the Mathematical Sciences,} Springer, vol. 5, 1989.

\bibitem{Agata} Bezubik A  and Strasburger A, \textsl{A New Form of the Spherical 
Expansion of Zonal Functions and Fourier Transforms of $SO(d)$-Finite Functions,} 
Symmetry, Integrability and Geometry: Methods and Applications, Vol. 2,  Paper 033, 8 pages. 2006.

\bibitem{IBV} Bogdanova, I;    Vandergheynst,P;   Antoine, J;  Jacques, L and  Morvidone, M,
\textsl{“Stereographic Wavelet Frames on the Sphere,”} Applied and
Computational Harmonic Analysis, vol. 19, no. 2, pp. 223–252, 2005.

\bibitem{TBU} Bulow, T., \textsl{Spherical diffusion for 3d surface smoothing,} 
IEEE Transactions on Pattern Analysis and Machine
Intelligence, vol. 26; 2004;  pp. 1650 – 1654, 2004.

\bibitem{RCP2}  Camporesi, R. \textsl{Harmonic Analysis and Propagators on 
Homogeneous  spaces.} Physics Reports 196. Nos 1 $\&$ 2, North-Holland; 1990; pp.1 - 34.

\bibitem{Heal}  Driscoll J. R; J. Dennis, J. and  Healy, M.,  \textsl{Computing Fourier transforms
and convolutions on the 2-sphere,} Adv. Appl. Math., vol. 15, no. 2, pp.202–250, 1994.
 
\bibitem{JOST} Jost, J., \textsl{Riemannian Geometry and geometric analysis,   
$5^{th}$   edition,}  \textit{Springer-Verlag}, Berlin Heidelberg, Germany, 2008.

\bibitem{Lee}  Lee, J. M.,  \textit{Introduction to smooth manifolds, Graduate Texts 
in Mathematics,} Springer Science and Business Media, New York, 2003.

\bibitem{MMT}  Morimoto, M., \textsl{Analytic Functionals on the sphere,  Translations 
of Mathematical Monographs \textbf{178},} American Mathematical Society, Providence, 1991.

\bibitem{OM} Omenyi, L. and Uchenna, U., \textsl{Global Analysis on Riemannian manifold,}  
The Australian  Journal of Mathematical Analysis and Applications 
\textbf{16}(02) (11), pp.1-17, 2019.

\bibitem{OM14}	Omenyi, L.,  On the second variation of the spectral 
zeta function of the Laplacian on homogeneous Riemannian manifolds:  
 Ph.D. Thesis, \textit{British Library},  2014,
[\texttt{Online: https://hdl.handle.net/2134/16167}]. 

\bibitem{STRA} Strasburger A., \textsl{A generalization of the Bochner identity, Exposition,} 
 Math., 1993, V.11, pp.153–157. 
 
\bibitem{Sze} Szekeres, P.,  \textsl{A Course in Modern Mathematical Physics : Groups,
Hilbert Space and Differential Geometry,} Cambridge University Press, 2004.

\bibitem{MWW}  Wogu, M. W., \textsl{Weyl Transforms, Heat Kernels, Green functions and 
Riemann Zeta functions on compact Lie groups. Modern Trends in Pseudo-differential
 operators,}  Operator Theory, Advances and Applications, Vol. 172, 
 Birkh$\ddot{a}$user Verlag, Basel/Switzerland, pp.67 - 85, 2006. 

\end{thebibliography}
\end{document}